\documentclass[11pt]{article}
\usepackage{amsthm}
\usepackage{latexsym,amssymb,amsmath}
\usepackage{graphicx,float,color,fancybox,shapepar,setspace,hyperref}
\usepackage{subfigure}
\usepackage{mathrsfs}
\usepackage{cite}
\usepackage{tabularx}
\usepackage{pgf,tikz}
\usepackage{caption}
\usetikzlibrary{arrows}
\usetikzlibrary{shapes.geometric}
\newtheorem{thm}{Theorem}[section]
\newtheorem{case}[]{Case}

\newtheorem{lem}[thm]{Lemma}
\newtheorem{prop}[thm]{Proposition}
\newtheorem{claim}[]{Claim}

\def\dfn#1{{\sl #1}}
\def\qed{\hfill\square}

\def\qed{ \hfill $\blacksquare$}

\begin{document}
\title{Some exact values on Ramsey numbers related to fans}

\author{Qinghong Zhao and Bing Wei\\
\small Department of Mathematics, University of Mississippi, University,  MS 38677,  USA}

\date{}
\maketitle	
\begin{abstract}
For two given graphs $F$ and $H$, the Ramsey number $R(F,H)$ is the smallest integer $N$ such that any red-blue edge-coloring of the complete graph $K_N$ contains a red $F$ or a blue $H$. When $F=H$, we simply write $R_2(H)$. For an positive integer $n$, let $K_{1,n}$ be a star with $n+1$ vertices, $F_n$ be a fan with $2n+1$ vertices consisting of $n$ triangles sharing one common vertex, and $nK_3$ be a graph with $3n$ vertices obtained from the disjoint union of $n$ triangles. In 1975, Burr, Erd\H{o}s and Spencer \cite{B} proved that $R_2(nK_3)=5n$ for $n\ge2$. However, determining the exact value of $R_2(F_n)$ is notoriously difficult. So far, only $R_2(F_2)=9$ has been proved.   Notice that both $F_n$ and $nK_3$ contain $n$ triangles and $|V(F_n)|<|V(nK_3)|$ for all $n\ge 2$. Chen, Yu and Zhao (2021) speculated that $R_2(F_n)\le R_2(nK_3)=5n$ for  $n$ sufficiently large. In this paper, we first prove that $R(K_{1,n},F_n)=3n-\varepsilon$ for $n\ge1$, where $\varepsilon=0$ if $n$ is odd and $\varepsilon=1$ if $n$ is even. Applying the exact values of $R(K_{1,n},F_n)$, we will confirm  $R_2(F_n)\le 5n$  for $n=3$ by showing that $R_2(F_3)=14$.

\end{abstract}	
\footnotetext{Email-address: qhzhao91@gmail.com (Q. Zhao) and bwei@olemiss.edu (B. Wei).}	
	
\section{Introduction}
 All graphs considered are finite, simple and undirected. Given a graph $G$, we denote by $V(G)$ the vertex set of $G$ and by $|V(G)|$ the number of vertices in $V(G)$. For a vertex $v\in V(G)$, let $N_G(v)$ denote the set of neighbors of $v$ in $G$. The degree of $v$ in $G$ is denoted by $d_G(v)$, that is, $d_G(v)=|N_G(v)|$. For a subset $S\subseteq V(G)$, let $G[S]$ denote the subgraph induced by the vertices of $S$, and we simply write $G-S$ as $G[V(G)-S]$. We use $C_n, T_n$ and $K_n$ to denote the cycle, tree and complete graph or clique on $n$ vertices, respectively. Given $k$ disjoint graphs $G_1,\ldots,G_k$, $G_1\cup\cdots\cup G_k$ denotes their disjoint union. In particular, if $G=G_1=\cdots=G_k$, we simply write $kG$. $G_1+G_2$ denotes the graph obtained from $G_1\cup G_2$ by joining every vertex in $V(G_1)$ to every vertex in $V(G_2)$. A star $K_{1,n}$ is $\{v\}+nK_1$, a fan $F_n$ is $\{v\}+nK_2$ and a book $B_n$ is $K_2+nK_1$, where the vertex $v$ is called the center of $K_{1,n}$ and $F_n$. For any integer $k\ge1$, we define $[k]=\{1,\ldots,k\}$. Given a complete graph whose edges are colored with red and blue, we write $R$ and $B$ for the graphs consisting of all red edges and blue edges, respectively. Given disjoint subsets $X,Y\subseteq V(G)$, if each vertex in $X$ is adjacent to all vertices in $Y$ and all the edges between $X$ and $Y$ are colored with the same color, then we say that $X$ is \dfn{$mc$-adjacent} to $Y$, that is, $X$ is \dfn{blue-adjacent} to $Y$ if all the edges between $X$ and $Y$ are colored with blue. 

Given two graphs $F$ and $H$, the Ramsey number $R(F,H)$ is the minimum integer $N$ such that any 2-edge-coloring of $K_N$ with colors red and blue yields a red $F$ or a blue $H$. Let $R(H,H)=R_2(H)$ be the diagonal Ramsey number. Then $R(F,H)$ is called off-diagonal Ramsey number when $F\neq H$. If both $F$ and $H$ are complete graphs, then $R(F,H)$ is usually called the classical Ramsey number as well. However, finding the classical Ramsey number is far from trivial in general. Therefore, it is natural to consider the generalized Ramsey numbers of graphs other than complete graphs. Chv\'atal and Harary \cite{CHV,CHVH1,CHVH2,CHVH3} first studied the generalized Ramsey numbers and a famous early results of Chv\'atal \cite{CHV} showed that $R(T_n,K_m)=(n-1)(m-1)+1$ for all positive integers $m$ and $n$. Determining the Ramsey numbers of trees versus other graphs is also a hot topic in graph theory. In 2015, Zhang, Broersma and Chen \cite{ZB} studied the Ramsey numbers of stars versus fans and proved the following result.
\begin{thm}[\cite{ZB}]\label{thmZBC}
$R(K_{1,n}, F_m) = 2n+1$ for all integers $n\ge m^2-m$ and $m\neq 3,4,5$, and this lower bound is the best possible. Moreover, $R(K_{1,n}, F_m)=2n+1$ for $n\ge 6m-7$ and $m=3,4,5$.
\end{thm}
\noindent
Since the lower bound $n\ge m^2-m$ and $m\neq 3,4,5$ in Theorem \ref{thmZBC} is the best possible, it is easily seen that $R(K_{1,n}, F_m)\ge 2n+2$ when $n\le m^2-m-1$. However, there are very few results on the exact values, especially for the case when $n=m$. In this paper, we first study the Ramsey numbers of $K_{1,n}$ versus $F_n$ and obtain the following result.
\begin{thm}\label{thm1}
$R(K_{1,n},F_n)=3n-\varepsilon$ for $n\ge1$, where $\varepsilon=0$ if $n$ is odd and $\varepsilon=1$ if $n$ is even.
\end{thm}
\noindent
Other results on Ramsey numbers concerning trees can be found in \cite{Ba,Bas,Bask,CZZ,CZZ1,CZZ2,GV}.
\vskip 0.1cm
The Ramsey numbers of fans versus fans have also been widely studied so far. In 1991, Li and Rousseau \cite{LR} showed that $4n+1\le R(F_m,F_n)\le4n+4m-2$ for $n\ge m\ge1$ and $R(F_1,F_n)=4n+1$ for $n\ge2$. Later, Lin and Li \cite{LL} improved the general upper bound as $R(F_m,F_n)\le4n+2m$ for $n\ge m\ge2$ and proved that $R(F_2,F_n)=4n+1$ for $n\ge2$. The latter result implies that $R_2(F_2)=9$.  However, the exact values of $R_2(F_n)$ for all $n\ge3$ are still unknown. For more related results on $R(F_m,F_n)$, see \cite{LLD,Z}. 

It is worth noting that $nK_3, F_n$ and $B_n$ are three graphs containing $n$ triangles with exactly zero, one and two vertices in common, respectively, and $|V(B_n)|\le|V(F_n)|\le|V(nK_3)|$. Thus the relationship among Ramsey numbers of such three graphs has received extensively attention and tremendous progresses on this topic have been made in recent years. In 1975, Burr, Erd\H{o}s and Spencer \cite{B} proved that $R_2(nK_3)=5n$ for $n\ge2$. Later, Rousseau and Sheehan \cite{RS} showed that $R_2(B_n)\le 4n+2$ for all $n$ and the bound is tight for infinitely many values of $n$. This shows that $R_2(B_n)\le R_2(nK_3)$ for $n\ge2$. Recently, Chen, Yu and Zhao \cite{CYZ} proved that $9n/2-5\le R_2(F_n)\le 11n/2+6$ for all $n\ge1$, which implies $R_2(B_n)<R_2(F_n)$ for sufficiently large $n$. Therefore, Chen et al. \cite{CYZ} believe that $R_2(F_n)\le R_2(nK_3)=5n$ for $n$ sufficiently large, even though they are unable to verify this. More recently, Dvo\v{r}\'ak and Metrebian \cite{DM} improved their upper bound to  $31n/6+15$ for all $n\ge1$. Note that $R_2(F_2)=9<R_2(2K_3)=10$. In this paper, we confirm $R_2(F_n)\le R_2(nK_3)=5n$ for $n=3$ by proving the following result.

\begin{thm}\label{thm2}
$R_2(F_3)=14$.
\end{thm}

For more information on Ramsey numbers, we refer the readers to two excellent surveys \cite{CFS,R}. We will prove Theorem \ref{thm1} in Section 2 and give the proof of Theorem \ref{thm2} in Section 3 after first show two structural lemmas.

\section{Proof of Theorem \ref{thm1} }
We first list three theorems that shall be applied in the proofs of Theorems \ref{thm1} and \ref{thm2}.
\begin{thm}[Hall's Theorem\cite{H}]\label{HH}
A bipartite graph $G$ with bipartition $X, Y$ has a matching that saturates $X$ if and only if $|N_G(S)|\ge |S|$ for all $S\subseteq X$.
\end{thm}
\begin{thm}[\cite{CL}]\label{CL} 
$R(K_{1,n},nK_2)=2n$ for all $n\ge1$.
\end{thm}
\begin{thm}[\cite{LL}]\label{LL} 
Let $m$ and $n$ be positive integers. Then $R(F_m,nK_2)=max\{m,n\}+m+n$.
\end{thm}

Now, we start to prove Theorem \ref{thm1}. For $n\ge1$ is odd ($resp.$, even), we first take a complete graph $H$ with $2n-1$ ($resp.$, $2n-2$) vertices in which each vertex has $n-1$ red neighbors and $n-1$ ($resp.$, $n-2$) blue neighbors, then let $G^{l}$ be a complete graph obtained from the join of a red $K_n$ and the graph $H$, and all the edges between them are colored with blue. Then, $|G^l|=3n-1-\varepsilon$, where $\varepsilon=0$ if $n$ is odd and $\varepsilon=1$ if $n$ is even. Since each vertex of $G^{l}$ has $n-1$ red neighbors and $2n-1$ ($resp.$, $2n-2$) blue neighbors, $G^{l}$ contains neither a red $K_{1,n}$ nor blue $F_n$.  Therefore, $R(K_{1,n},F_n)\ge3n-\varepsilon$ for $n\ge1$.

 Now, we will show that $R(K_{1,n},F_n)\le3n-\varepsilon$. For all $n\ge1$, let $G$ be a complete graph with $3n-\varepsilon$ vertices such that the edges of $G$ are colored by red and blue, where $\varepsilon=0$ if $n$ is odd and $\varepsilon=1$ if $n$ is even. Suppose that $G$ contains neither red $K_{1,n}$ nor blue $F_n$. If $n$ is odd, then by Theorem \ref{CL}, we see that $d_B(v)\le 2n-1$ for any $v\in V(G)$, implying that $d_R(v)\ge n$, which gives us a red $K_{1,n}$, a contradiction. If $n$ is even, then $|V(G)|=3n-1$ is odd. To avoid a red $K_{1,n}$, we see that $d_B(v)\ge 2n-1$ for any $v\in V(G)$. As both $|V(G)|$  and $2n-1$ are odd, there exists a vertex $u\in V(G)$ such that $d_B(u)\ge2n$. Thus by Theorem \ref{CL}, we can obtain a blue $F_n$ with center $u$, a contradiction. Therefore, Theorem \ref{thm1} follows.
 \vskip 0.2cm
 {\bf Remark.} The extremal graph $G^l$ on $3n-1-\varepsilon$ vertices without red $K_{1,n}$ and blue $F_n$ is constructed for any $n\ge 1$ in the proof above. Noting that $3n-1-\varepsilon$ is even while $n$ is even or odd, we can construct extremal graphs other than $G^l$.
 
 \section{Proof of Theorem \ref{thm2} }
 We first show two lemmas which play very important role in the proof of Theorem \ref{thm2}.

\begin{lem}\label{lemz} 
Let $G$ be a complete graph with 14 vertices such that the edges of $G$ are colored by red and blue without monochromatic copy of $F_3$. Then $G$ contains no monochromatic copy of $K_4+2K_1$.
\end{lem}
\begin{proof}
Without loss of generality, suppose that $G$ contains a blue $H=K_4+2K_1$. Set $V(H)=\{u_1,\ldots,u_6\}$, $V(2K_1)=\{u_1,u_6\}$ and $K=V(G)-V(H)=\{v_1,\dots,v_8\}$. Then we see that $|N_R(v_i)\cap \{u_1,\ldots,u_5\}|\ge4$ for $i\in[8]$. If $G[K]$ contains a red $K_{1,3}$, set $K_{1,3}=\{w:x,y,z\}$ with center $w$, then by Hall's theorem, there exists a red 
$3K_2$ between $N_R(w)\cap \{u_1,\ldots,u_5\}$ and $\{x,y,z\}$, which leads to a red  $F_3$ with center $w$ in $G[V(K_{1,3})\cup \{u_1,\ldots, u_5\}]$, a contradiction. It follows that $|N_R(v_i)\cap K|\le2$ for $i\in[8]$. If there is some $i\in [8]$, say $i=1$, such that $|N_R(v_1)\cap K|\le1$ and $\{v_3,\dots,v_8\}\subseteq N_B(v_1)$. Then by Theorem \ref{CL}, there is a blue $3K_2$ in $G[\{v_3,\dots,v_8\}]$, which forms a blue $F_3$ together with $v_1$, a contradiction. Therefore, we may assume that $|N_R(v_i)\cap K|=2$ for $i\in[8]$. Thus, $G[K]$ has a red $2$-factor consisting of a red $C_8$ or $2C_4$ or $C_5\cup C_3$. Since $G$ has no red $F_3$, we have $|N_B(u_i)\cap K|\ge 2$ for $i\in[6]$. Since $|K|=8$, there exists a vertex $v' \in K$ such that $|N_B(v')\cap \{u_1,\dots,u_5\}|\ge2$, which leads to a blue $F_3$ with center in $\{u_1,\ldots,u_5\}$, yielding a contradiction. Therefore, Lemma \ref{lemz} holds.
\end{proof}

\begin{lem}\label{lemzz}
Let $G$ be a complete graph with 14 vertices such that the edges of $G$ are colored by blue and red without monochromatic copy of $F_3$. If $d_B(v)\le7$ and $d_R(v)\le7$ for any $v\in V(G)$, then $G$ contains no monochromatic copy of $K_5$. 
\end{lem}
\begin{proof}
Without loss of generality, suppose that $G$ contains a blue $H=K_5$. Set $V(H)=\{u_1,\dots,\\u_5\}$ and $K=V(G)-V(H)=\{v_1,\dots,v_9\}$. Let $V_k$ be a vertex set of $K$ in which each vertex is blue-adjacent to $k$ vertices of $V(H)$. By Lemma \ref{lemz}, it is easily seen that $|V_k|=0$ for all $k\ge4$. Since $d_B(v)\le7$ and $d_R(v)\le7$ for any $v\in V(G)$, we have $d_R(v)\ge 6$ and $d_B(v)\ge 6$ as $|V(G)|=14$.
We first prove the following five properties. 
\begin{itemize}
\item[(1)] $N_B(u_i)\cap K$ induces a red clique if $|N_B(u_i)\cap K|\ge 2$ for $i\in [5]$;
\item[(2)] $V_k\not=\emptyset$ for some $2\le k\le 3$;
\item[(3)] for any two vertices in $K$, say  $v_1$ and $v_2$, $|(N_B(v_1)\cup N_B(v_2))\cap V(H)|=2$ if  $|N_B(v_1)\cap N_B(v_2)\cap V(H)|\ge 1$, $|N_B(v_1)\cap V(H)|\ge 2$ and $|N_B(v_2)\cap V(H)|\ge 2$;
\item[(4)] for any three vertices, say $v_1, v_2, v_3$ in $K$, $|N_B(v_1)\cap V(H)|+|N_B(v_2)\cap V(H)|+|N_B(v_3)\cap V(H)|\le 7$;
\item[(5)] $|N_R(v_i)\cap K|\not=4$ for $i\in [9]$.
\end{itemize}
\begin{proof}
For (1), if there is a blue edge in $N_B(u_i)\cap K$ for some $i\in [5]$, then $G$ contains a blue $F_3$ with center in $V(H)$, a contradiction. 

Since $|N_B(v)\cap K|\ge 2$ for any $v\in V(H)$, (2) holds by $|K|=9$. 

For (3), if $|(N_B(v_1)\cup N_B(v_2))\cap V(H)|\ge 3$, we can easily derive that $V(H)\cup \{v_1,v_2\}$ induces a blue $F_3$ as $|N_B(v_1)\cap N_B(v_2)\cap V(H)|\ge 1$, a contradiction and thus (3) follows. 

For (4), since $|V_3|\le1$  by (3) and  $|V_k|=0$ for all $k\ge4$, (4) holds.

For (5), suppose to the contrary that $|N_R(v_1)\cap K|=4$ and $\{v_2,\dots,v_5\}\subseteq N_R(v_1)$. Set $C=\{v_2,\dots,v_5\}$ and $S=\{v_6,\dots,v_9\}$. Then $S\subseteq N_B(v_1)$. Since $d_R(v_1)\le 7$ and $d_B(v_1)\le 7$, we have $2\le |N_R(v_1)\cap V(H)|\le 3$. 
If $|N_R(v_1)\cap V(H)|=3$,  let $\{u_3,u_4,u_5\}\in N_R(v_1)$, then $u_iv_1\in B$ and $u_iv_j\in R$ by (1) for $i\in [2]$ and $u_j\in S$. Since $d_B(u_1)\ge  6$, there is a vertex in $C$, say $v_5$, such that $v_5u_1\in B$. By (3), $v_5u_j\in R$ for $j=3,4,5$ and by (4) there are a vertex in  $\{u_3,u_4,u_5\}$, say $u_3$ and a vertex in $\{v_2,v_3,v_4\}$, say $v_4$ such that $u_3v_4\in R$. Thus, $v_iu_4\in B$ for $i=2,3$ to  to avoid a red $F_3$ with center $v_1$. Hence, $v_2v_3\in R$ by (1), which leads to a red $F_3$ in $G[\{v_1,\ldots,v_5\}\cup \{u_3,u_5\}]$, a contradiction.

If $|N_R(v_1)\cap V(H)|=2$, let $u_4v_1,u_5v_1\in R$ and $u_iv_1\in B$ for $i\in [3]$. Since $S\subseteq N_B(v_1)$, we have $S\subseteq N_R(u_i)$ for $i\in [3]$ by (1). To avoid a red $F_3$ with center in $S$,  $G[S]$ contains no red $K_{1,3}$. Moreover, to avoid a blue $F_3$ with center $v_1$, $G[S]$ contains no blue $2K_2$ as well. Therefore, $G[S]$ consists of a red $C_3$ and a blue $K_{1,3}$. Without loss of generality, we may assume that $V(C_3)=\{v_6,v_7,v_8\}$ and $v_6v_9,v_7v_9,v_8v_9\in B$. Since $d_B(u_i)\ge 6$ for $i\in [3]$, we may assume that $u_1v_2,u_2v_3,u_3v_4\in B$ by (3). Then $\{v_2,v_3,v_4\}$ is red-adjacent to $\{u_4,u_5\}$ by (3) and to avoid a red $F_3$ with center $v_1$, we have   $v_2v_5,v_3v_5,v_4v_5\in B$. To avoid a blue $F_3$ with center $v_5$, $N_R(v_j)\cap C\not=\emptyset$ for some $6\le j\le 8$. Without loss of generality, assume that $v_6v_i\in R$, where $v_i\in C$. Then by (3), there is a vertex in $\{u_1,u_2,u_3\}$, say $u_1$, such that $u_1v_i\in R$. Noting that $G[\{v_6,v_7,v_8\}]$ is a red triangle and $S\subseteq N_R(u_j)$ for $j\in [3]$, we can get a red $F_3$ with center $v_6$, a contradiction.
\end{proof}

Now, we turn to prove Lemma \ref{lemzz}. Let  $q=\max\{|N_R(v)\cap K|: v\in K\}$. Without loss of generality, we may assume that $|N_R(v_1)\cap K|=q$. Then $q\ge 3$ by applying Theorem \ref{thm1} to $n=3$ as there is no blue $F_3$ in $G[K]$, $q\not=4$ by (5) and $q\le 5$ as $|N_R(v_1)\cap V(H)|\ge 2$ by Lemma \ref{lemz}  and $d_R(v_1)\le 7$. If $q=3$, then $V_k=\emptyset$ for $k\ge 3$ as $d_R(v)\ge6$ for any $v\in V(G)$. Furthermore, by (2), we see that $V_2\neq\emptyset$, say $v_1\in V_2$. Let $N_R(v_1)\cap K=\{v_2,v_3,v_4\}$, $T=N_B(v_1)\cap K=\{v_5,\ldots,v_9\}$, $U_1=N_B(v_1)\cap V(H)=\{u_1,u_2\}$ and $U_2=N_R(v_1)\cap V(H)=\{u_3,u_4,u_5\}$.  Then  $U_1$ is red-adjacent to $T$ by (1). To avoid a blue $F_3$ with center $v_1$, $G[T]$ contains no blue $2K_2$, implying $d_{B[T]}(v)\le 1$ for some $v\in T$.  Since $|T|\ge 5$, there exists a red $K_{1,3}$ in $G[T]$ with center, say $v_5$. Let $v_iv_5\in R$ for $i=6,7,8$. As $q=3$ and $d_R(v_5)\ge 6$, $N_R(v_5)\cap \{u_3,u_4,u_5\}\not=\emptyset$, say $v_5u_3\in R$. To avoid a red $F_3$ with center $v_5$ in $G[\{v_5,\ldots,v_8\}\cup\{u_1,u_2,u_3\}]$, we have $v_iu_3\in B$ for $i=6,7,8$. Thus by (1), $G[\{v_5,v_6,v_7,v_8\}]$ is a red $K_4$, which will lead to a red $K_4+2K_1$ together with $\{u_1,u_2\}$, contrary to Lemma \ref{lemz}. Therefore, we may conclude that $q=5$.
 
Similar to the above discussion, let $S=N_R(v_1)\cap K=\{v_2,\ldots,v_6\}$, $T=N_B(v_1)\cap K=\{v_7,v_8,v_9\}$, $U_1=N_B(v_1)\cap V(H)=\{u_1,u_2,u_3\}$ and $U_2=N_R(v_1)\cap V(H)=\{u_4,u_5\}$.  Then  $U_1$ is red-adjacent to $T$ by (1). We first show that $G[S]$ is a blue $K_5$. Suppose there is a red edge, say $v_5v_6$, in $G[S]$. If there is a red edge between $U_2$ and $\{v_2,v_3, v_4\}$, say $v_4u_4$, then $\{v_2,v_3\}\subseteq N_B(u_5)$, which implies $v_2v_3\in R$, resulting in a red $F_3$ with center $v_1$, a contradiction. If $U_2$ is blue-adjacent to $\{v_2,v_3, v_4\}$, then $G[\{v_2,v_3, v_4\}]$ is a red clique by (1). Noting that $d_B(u_j)\le 7$ for $j=4,5$, we see that $U_2$ is red-adjacent to $\{v_5,v_6\}$, which again leads to a red $F_3$ with center $v_1$, a contradiction. Therefore, $S$ induces a blue $K_5$.

If $T$ induces a blue $K_3$, then by (1), there are at most 7 blue edges between $T$ and $U_2\cup S$. On the other hand, since $d_R(v_j)\le 7$ for $j=7,8,9$, there are at most $3\times 4=12$ red edges between $T$ and $U_2\cup S$. Thus there are at most $12+7=19<21$ edges between them,  yielding a contradiction.
 
If $G[T]$  contains only one red edge, say $v_7v_8$, then by (1) and (4), there are at most $7+4=11$ blue edges between $T$ and $U_2\cup S$. As $d_R(v_j)\le 7$ for $j=7,8,9$, there are at most $4+3\times 2=10$ red edges between $T$ and $U_2\cup S$. Since there are $3\times 7=21$ edges between $T$ and $U_2\cup S$, $\{v_7,v_8\}$ has to be blue-adjacent to $U_2$, implying $|N_R(v_7)\cap (V(G)-S)|=4$, contrary to (5) as $S$ induces a blue $K_5$. 
  
If $G[T]$ contains  two red edges,  then by (1) and (4), there are at most $7+4=11$ blue edges between $T$ and $U_2\cup S$. Again, as $d_R(v_j)\le 7$ for $j=7,8,9$, there are at most $3\times 2+2=8$ red edges between $T$ and $U_2\cup S$. Thus there are at most $11+8=19<21$ edges between them,  a contradiction.
  
   Finally, if $S$ induces a red $K_3$, then by (4), there are at most $7+6=13$ blue edges between $T$ and $U_2\cup S$. Similar to the above discussion, there are at most $3\times 2=6$ red edges between $T$ and $U_2\cup S$. Thus there are at most $13+6=19<21$ edges between them,  a contradiction. 
   
   Therefore, the proof of Lemma \ref{lemzz} is complete.
\end{proof}

The following result provides a lower bound for $R_2(F_n)$ when $n$ is odd.

\begin{prop}\label{prop}
$R_2(F_n)\ge 4n+2$ for odd $n\ge1$.
\end{prop}
\begin{proof}
Let $H$ be a complete graph with 5 vertices such that the edges of $H$ are colored by red and blue without monochromatic copy of triangle. For odd $n\ge1$, let $G$ be obtained by replacing four vertices of $H$ with two red $H_1=K_n$ and two blue $H_2=K_n$ such that red ($resp.$, blue) $K_n$ is not red ($resp.$, blue)-adjacent to red ($resp.$, blue) $K_n$ (see Fig. \ref{t1}. The solid lines are colored with red and the dashed lines are colored with blue). Clearly, $G$ contains no monochromatic copy of $F_n$ for odd $n\ge1$. Thus the statement follows.
\end{proof}
\begin{figure}[htp]
\begin{center}
	\def\r{4pt}
	\def\dy{1cm}
	\tikzset{c/.style={draw,circle,fill=black,minimum size=.2cm,inner sep=0pt,anchor=center},
		d/.style={draw,circle,fill=white,minimum size=1cm,inner sep=0pt, anchor=center}}
		\begin{tikzpicture}[font=\large]
		\pgfmathtruncatemacro{\Ncorners}{5}
		\node[draw, regular polygon,regular polygon sides=\Ncorners,minimum size=3cm] 
		(poly\Ncorners) {};
		\node[regular polygon,regular polygon sides=\Ncorners,minimum size=3.5cm] 
		(outerpoly\Ncorners) {};
		\foreach\x in {1}{
			\node[d] (poly\Ncorners-\x) at (poly\Ncorners.corner \x){$H_1$};
		}	
		\foreach\x in {4}{
			\node[d] (poly\Ncorners-\x) at (poly\Ncorners.corner \x){$H_1$};
		}	
		\foreach\x in {2,3}{
			\node[d] (poly\Ncorners-\x) at (poly\Ncorners.corner \x){$H_2$};
			\node[d] (poly\Ncorners-\Ncorners) at (poly\Ncorners.corner \Ncorners){$\bullet$};
		}
		\foreach\X in {1,...,\Ncorners}
		{\foreach\Y in {1,...,\Ncorners}{
				\pgfmathtruncatemacro{\Z}{abs(mod(abs(\Ncorners+\X-\Y),\Ncorners)-2)}
				\ifnum\Z=0
				\draw [line width=2mm, dashed](poly\Ncorners-\X) -- (poly\Ncorners-\Y);
				\fi}}
		\foreach\X in {1,...,\Ncorners}
		{\foreach\Y in {1,...,\Ncorners}{
				\pgfmathtruncatemacro{\G}{abs(mod(abs(\Ncorners+\X-\Y),\Ncorners)-4)}
				\ifnum\G=0
				\draw [line width=2.5mm](poly\Ncorners-\X) -- (poly\Ncorners-\Y);
				\fi}}
		\end{tikzpicture}
\end{center}
\caption{A construction without a monochromatic copy of $F_n$ for odd $n\ge1$.}
\label{t1}
\end{figure}
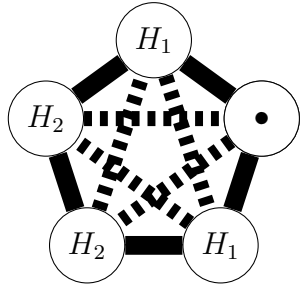
\textbf{Proof of Theorem \ref{thm2}.}  By Proposition \ref{prop}, it suffices to show that  $R_2(F_3)\le14$. Let $G$ be a complete graph with 14 vertices such that the edges of $G$ are colored by red and blue. Suppose that $G$ contains no monochromatic copy of $F_3$.  Let $m=max\{d_R(v), d_B(v)\}$ for any $v\in V(G)$. Then by Theorem \ref{LL}, we have $8\ge m\ge7$. Without loss of generality, we may assume that $u$ is a vertex in $V(G)$ such that $d_R(u)=m$. We distinguish two cases.
\begin{case}\label{case1}
$d_R(u)=8$.
\end{case}
Let $M_r$ denote the maximum red matching in $G[N_R(u)]$. Clearly, $1\le|M_r|\le2$. Furthermore, by Lemma \ref{lemz}, we see that $|M_r|=2$. Set $N_R(u)=\{u_1,\dots,u_8\}$. Without loss of generality, we may assume that $u_1u_2,u_3u_4\in M_r$. Then $\{u_5,\ldots,u_8\}$ induces a blue $K_4$.  If there exists a red $2K_2$ between $M_r$ and $\{u_5,\ldots,u_8\}$,  we may assume that $u_1u_5, u_3u_6\in R$ as $|M_r|=2$, Then $u_2$ and $u_4$ are blue-adjacent to $\{u_4,u_6,u_7,u_8\}$ and $\{u_2,u_5,u_7,u_8\}$, respectively. Thus by Lemma \ref{lemz}, we have $u_2u_5, u_4u_6\in R$, and hence by symmetry, $u_1$ and $u_3$ are blue-adjacent to $\{u_6,u_7,u_8\}$ and $\{u_5,u_7,u_8\}$, respectively.  This leads to a blue $F_3$ with center $u_8$, which is a contradiction. If  If there is no red $2K_2$ between $M_r$ and $\{u_5,\ldots,u_8\}$, by Lemma \ref{lemz},  
we can find a vertex in $\{u_5,\ldots,u_8\}$, say $u_5$, which is red-adjacent to at least three vertices in $V(M_r)$, say $u_1,u_2$ and $u_3$. It follows that $\{u_1,\ldots,u_4\}$ is blue-adjacent to $\{u_6,u_7,u_8\}$. In order to avoid blue $F_3$ with center in $\{u_6,u_7,u_8\}$, $\{u_1,\ldots,u_4\}$ must induce a red $K_4$, which forms a red $K_4+2K_1$ together with $u_5$ and $u$, contrary to Lemma \ref{lemz}.  This completes the proof of Case \ref{case1}.

\begin{case}\label{case2}
$d_R(u)=7$.
\end{case}
Set $N_R(u)=\{u_1,\ldots,u_7\}$ and $T=N_B(u)=\{v_1,\ldots,v_6\}$. We first prove the following claim.
\begin{claim}\label{claim2}
$G[N_R(u)]$ contains no blue $K_3+3K_1$.
\end{claim}
\begin{proof}
Suppose not. Let $V(K_3)=\{u_1,u_2,u_3\}$ and $V(3K_1)=\{u_4,u_5,u_6\}$. Then by Lemma \ref{lemzz}, we can derive that  $\{u_4,u_5,u_6\}$ induces a red $K_3$, which further implies $N_B(u_7)\cap \{u_4,u_5,u_6\}\not=\emptyset$, say $u_6u_7\in B$. Hence, to avoid a blue $F_3$ with center $u_i$, we have $u_7u_i\in R$ for $i\in [3]$. Moreover, we have $|N_B(u_i)\cap T|\ge 1$ as $d_B(u_i)\ge 6$. If $|N_B(u_i)\cap T|\ge 2$ for some $i\in [3]$, without loss of generality, assume that $u_1v_1,u_1v_2\in B$. To avoid a blue  $F_3$ with center $u_1$, $\{v_1,v_2\}$ is red-adjacent to $\{u_4,u_5,u_6\}$ and $v_1v_2\in R$, implying $\{v_1,v_2,u_4,u_5,u_6\}$ induces a red $K_5$,  contrary to Lemma \ref{lemzz}. Thus, $|N_B(u_i)\cap T|= 1$ for $i\in [3]$. 

By Theorem \ref{CL}, $G[T]$ contains a red $K_{1,3}$ as  $G[T]$ contains no blue $3K_2$. Without loss of generality, we assume that $K_{1,3}=\{v_4: v_1,v_2,v_3\}$ with center $v_4$. 
\begin{figure}[htp]
\begin{center}
	\includegraphics[scale=.6]{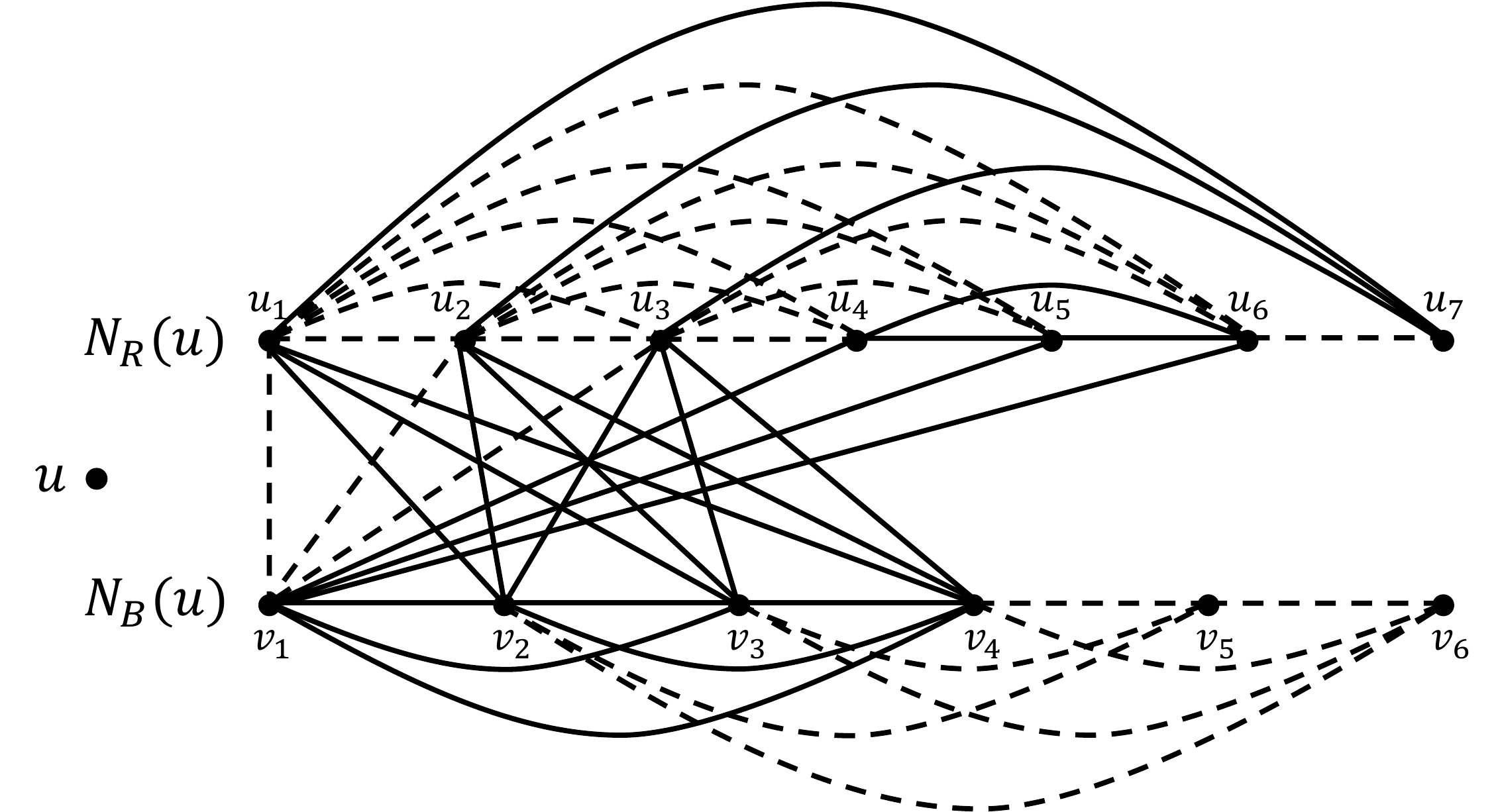}
\end{center}
\caption{ $\{v_1, v_2,v_3, v_4, u_1,u_2,u_3\}$ induces a red $K_3+4K_1$ and  $v_1u_j\in R$ for $j=4,5,6$.}
\label{t2}
\end{figure}

If $v_4\notin N_B(u_i)\cap T$ for any $i\in [3]$, then $v_4\in N_R(u_i)$. Noting that $|N_R(u_i)\cap \{v_1,v_2,v_3\}|\ge 2$ as  $|N_B(u_i)\cap T|=1$ for $i\in [3]$, by Hall's theorem there is a red $3K_2$ between $\{u_1,u_2,u_3\}$ and $\{v_1,v_2,v_3\}$ when $N_R(v_i)\cap \{u_1,u_2,u_3\}\not=\emptyset$ for any $i\in [3]$, which leads to a red $F_3$ with center $v_4$, a contradiction. Thus there exists some $i\in [3]$, say $i=1$, such that $N_R(v_1)\cap \{u_1,u_2,u_3\}=\emptyset$. Then we have $v_1\in N_B(u_i)$ and $\{v_2,v_3,\cdots, v_6\}\subseteq N_R(u_i)$ for $i\in [3]$, implying $\{v_5,v_6\}\subseteq N_B(v_4)$ to avoid a red $F_3$ with center $v_4$. Noting that $\{u_1,u_7,u\}$ induces a red triangle, there is no red $2K_2$ in $G[\{v_2,v_3,\cdots, v_6\}]$, implying $v_5v_6\in B$ and $\{v_5,v_6\}$ is blue-adjacent to $\{v_2,v_3\}$.  Thus, $v_2v_3\in R$ by applying Lemma \ref{lemzz} to $\{v_2,v_3,v_5,v_6, u\}$. Since there is no blue $3K_2$ in $G[T]$, we have $v_1v_2, v_1v_3\in R$. Notice that $\{v_1, v_2,v_3, v_4, u_1,u_2,u_3\}$ induces a red $K_3+4K_1$, $u_7u_i\in R$ for $i\in [3]$ and  $v_1u_j\in R$ for $j=4,5,6$ (see Fig. \ref{t2}. The solid lines are colored with red and the dashed lines are colored with blue). To avoid a red $F_3$ with center in $\{v_1,\ldots,v_4\}$, $\{u_4,\ldots,u_7\}$ is blue-adjacent to $\{v_2,v_3,v_4\}$. To avoid a blue $F_3$ with center $u_7$, we may assume that $u_5u_7\in R$. Since $d_B(u_5)\le 7$, we have $N_R(u_5)\cap \{v_5,v_6\}\not=\emptyset$, say $v_5u_5\in R$. Then $\{u, u_4,u_5,u_6,u_7, v_1,v_5\}$ induces a red $F_3$ with center $u_5$ whenever $v_5u_7\in R$ or $\{u, u_7, v_2,\ldots,v_6\}$ induces a blue $F_3$ with $v_5$ whenever $v_5u_7\in B$, which is a contradiction.

Now, we may assume that the center of any induced red $K_{1,3}$ in $G[T]$ has a blue neighbor in $\{u_1,u_2,u_3\}$. Let $v_4u_1\in B$ and $T'=T-\{v_4\}$. Then $T'\subseteq N_R(u_1)$ and there is no red $2K_2$ in $G[T']$ as $\{u_1,u_7,u\}$ induces a red $K_3$. If there is a vertex, say $x$, in $T'$ such that $|N_R(x)\cap T|\ge 3$. Without loss of generality, we may assume that $u_2x\in B$ as $x$ is the center of a red $K_{1,3}$. Recall that $|N_B(u_i)\cap T|=1$ for $i\in [3]$, implying either $\{u_3,u_2\}\subseteq N_R(v_4)$  or $\{u_3,u_1\}\subseteq N_R(x)$, which leads to $d_R(v_4)\ge 8$ or $d_R(x)\ge 8$ as $\{v_4,x\}$ is red-adjacent to $\{u_4,u_5,u_6\}$ to avoid a blue $F_3$, a contradiction. Thus $|N_R(x)\cap T|\le 2$ for any $x\in T'$. Without loss of generality, we may assume that $v_1v_2,v_1v_3\in B$. When $v_2v_3\in R$, then $\{v_5v_6,v_1v_5,v_1v_6\}\subseteq B$ to avoid a red $2K_2$ in $G[T']$. Since $|N_R(v_2)\cap T|\le 2$, $\{u,v_1, v_2,v_5,v_6\}$ induces a blue $K_5$, a contradiction. When $v_2v_3\in B$, since $|N_R(v_i)\cap T|\le 2$ for $i\in [3]$, to avoid a blue $K_5$, $v_5$ and $v_6$ must have different red neighbors in $\{v_1,v_2,v_3\}$, yielding a red $2K_2$ in $G[T']$, a contradiction again.
\end{proof}

We are now ready to finish the proof for Case \ref{case2}. Let $M_r$ denote the maximum red matching in $G[N_R(u)]$. By similar arguments as in Case \ref{case1}, we have $|M_r|=2$. Without loss of generality, assume that $u_1u_2,u_3u_4\in M_r$. Then $\{u_5,u_6,u_7\}$ induces a blue $K_3$. Let $S= \{u_1,u_2,u_3,u_4\}$ and assume that, in $S$, $u_1$ has the minimum number of red neighbors in $\{u_5,u_6,u_7\}$, denoted by $d$. 
\begin{figure}[htp]
\begin{center}
	\includegraphics[scale=.55]{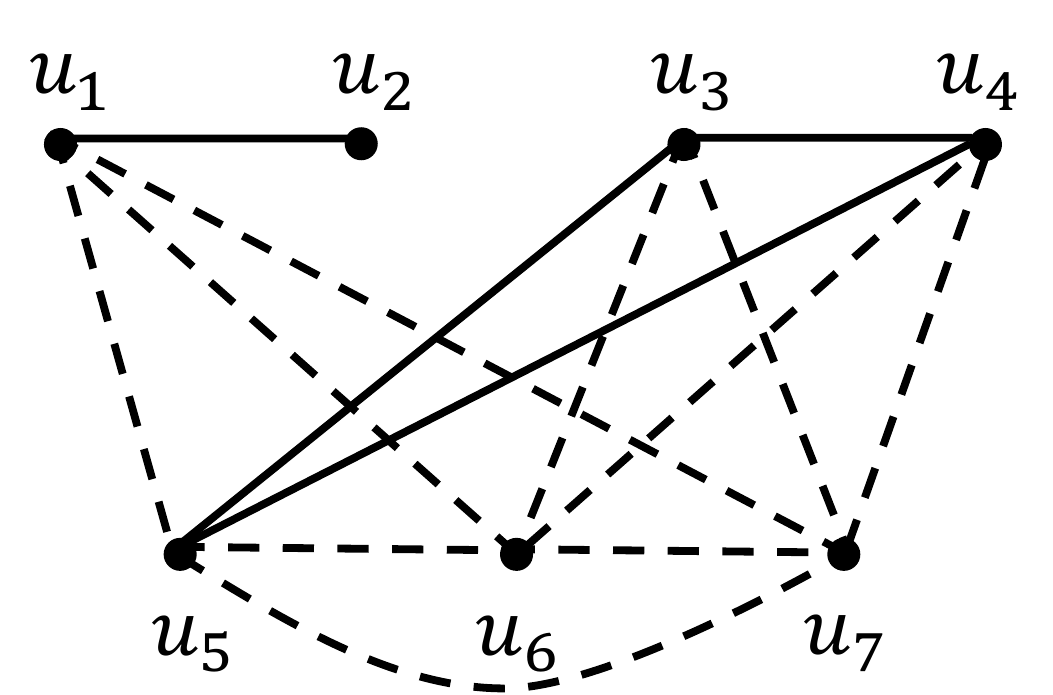}
\end{center}
\caption{Illustration for Case \ref{case2} when $d=0$.}
\label{t3}
\end{figure}

If $d\ge 1$, as $|M_r|=2$, we have $|N_R(u_i)\cap \{u_5,u_6,u_7\}|=1$ for $i\in[4]$ and there exist $w_j\in \{u_5,u_6,u_7\}$ such that $w_j\in N_R(u_j)\cap N_R(u_{j+1})$ for $j=1,3$.
When $w_1\not=w_3$, we may assume that $u_7\notin \{w_1,w_3\}$ and $u_7$ is blue-adjacent to $S$. As  $|M_r|=2$, $G[S]$ contains a blue $C_4$, which leads to a blue $F_3$ with center $u_7$, a contradiction. When $w_1=w_3$, say $u_5=w_1=w_3$, then $\{u_6,u_7\}$ is blue-adjacent to $S$. Clearly, to avoid a blue $F_3$ with center $u_6$ or $u_7$,
there is no blue $2K_2$ in $G[S]$. Thus $G[S]$ contains a red triangle, which forms a red $K_5$ together with $u_5$ and $u$, contrary to Lemma \ref{lemzz}.

If $d=0$, then $\{u_5,u_6,u_7\}\subseteq N_B(u_1)$. When there is a vertex in $\{u_3,u_4\}$, say $u_3$, such that $\{u_5,u_6,u_7\}\subseteq N_B(u_3)$, by Lemma \ref{lemzz}, we have $u_1u_3\in R$.  By Claim \ref{claim2}, $N_R(u_j)\cap \{u_5,u_6,u_7\}\not=\emptyset$ for $j=2,4$. As $|M_r|=2$, there exists exactly one vertex, say $u_5$, such that $u_5$ is red-adjacent to $\{u_2,u_4\}$. Thus, $\{u_6,u_7\}\subseteq N_B(u_j)$, $j=2,4$. To avoid a blue $F_3$ with center $u_6$ or $u_7$, we see that $S\cup \{ u\}$ induces a red $K_5$, a contradiction to Lemma \ref{lemzz}. Hence, $\{u_5,u_6,u_7\}\cap N_R(u_i)\not=\emptyset$ for $i=3,4$. As $|M_r|=2$, there exists exactly one vertex, say $u_5$, such that $u_5$ is red-adjacent to $\{u_3,u_4\}$, implying $\{u_6,u_7\}\subseteq N_B(u_i)$ for $i=3,4$ (see Fig. \ref{t3}. The solid lines are colored with red and the dashed lines are colored with blue). When $N_R(u_2)\cap \{u_6,u_7\}\not=\emptyset$, then $u_1$ is blue-adjacent to $\{u_3, u_4\}$ and $G[N_R(u)]-\{u_2\}$ contains a blue $K_3+3K_1$ with $G[\{u_1,u_6,u_7\}]$ as the blue $K_3$, contrary to Claim \ref{claim2}. When $N_R(u_2)\cap \{u_6,u_7\}=\emptyset$, then  $\{u_6,u_7\}\subseteq N_B(u_2)$ and to avoid a blue $F_3$ with center $u_7$, we have $u_2u_i\in R$ for $i=3,4$, implying $u_2u_5\in B$, otherwise  $\{u_2,u_3,u_4,u_5\}\cup \{ u\}$ induces a red $K_5$. Finally, to avoid a blue $F_3$ with center $u_7$, we have $u_1u_3,u_1u_4\in R$, implying that  $S\cup \{ u\}$ induces a red $K_5$,  contrary to Lemma \ref{lemzz}. This completes the proof of Case \ref{case2}.

 Hence, we complete the proof of Theorem \ref{thm2}.\qed

\end{document}